\documentclass{article}

\usepackage{pdfsync}

\usepackage[colorlinks=true,linkcolor=blue,citecolor=blue]{hyperref}

\usepackage{latexsym,epsfig,tabularx,amssymb,enumerate}
\usepackage{algorithm}
\usepackage{algpseudocode}
\algrenewcommand{\algorithmicrequire}{\textbf{Input:}}
\algrenewcommand{\algorithmicensure}{\textbf{Output:}}

\usepackage{microtype}

\newcommand{\showkeyslabelformat}[1]{%
}
\usepackage[notref,notcite]{showkeys}

\usepackage{graphicx}
\usepackage{color}
\usepackage{amsmath,amsfonts,amsthm,bm}
\usepackage[a4paper]{geometry}

\usepackage{lipsum}
\usepackage{tikz}
\usepackage{pgfplots}
\usepackage{pgfplotstable}
\usepackage{soul}

\newcommand{\Z}{\mathbb{Z}} 
\newcommand{\C}{\mathbb{C}} 
\newcommand{\N}{\mathbb{N}} 

\newcommand{\imagunit}{\mathrm{i}}
\newcommand{\twopii}{2 \pi \imagunit \,}

\newcommand{\sh}{\mathrm{sh}}

\DeclareSymbolFont{bbold}{U}{bbold}{m}{n}
\DeclareSymbolFontAlphabet{\mathbbold}{bbold}
\newcommand{\ind}[1]{\mathbbold{1}_{#1}}

\newcommand*{\link}[1]{(\ref{#1})}                                      
\newcommand*{\abs}[1]{\left| #1 \right|}                                
\newcommand*{\nach}{\rightarrow}                                        
\newcommand*{\norm}[1]{\left\| #1 \right\|}                             
\newcommand{\ie}{i.e.}
\newcommand{\eg}{e.g.}

\theoremstyle{plain}
\newtheorem{theorem}{Theorem}
\newtheorem{lemma}{Lemma}

\newtheorem{corollary}{Corollary}

\theoremstyle{definition}

\theoremstyle{remark}

\newcommand{\setu}{\mathfrak{u}}

\newcommand{\rd}{\,\mathrm{d}} 

\def\citep#1#2{\cite[{#1}]{#2}}

\newcommand{\bszero}{{\boldsymbol{0}}} 
\newcommand{\bst}{{\boldsymbol{t}}}    
\newcommand{\bsh}{{\boldsymbol{h}}}    
\newcommand{\bsk}{{\boldsymbol{k}}}    
\newcommand{\bsl}{{\boldsymbol{\ell}}}    
\newcommand{\bsell}{{\boldsymbol{\ell}}}    
\newcommand{\bsx}{{\boldsymbol{x}}}    
\newcommand{\bsy}{{\boldsymbol{y}}}    
\newcommand{\bsz}{{\boldsymbol{z}}}    
\newcommand{\bsDelta}{{\boldsymbol{\Delta}}}    
\newcommand{\bsgamma}{{\boldsymbol{\gamma}}}    
\newcommand{\bssigma}{{\boldsymbol{\sigma}}}    
\newcommand{\bbR}{{\mathbb{R}}}
\newcommand{\calO}{{\mathcal{O}}}

\usepackage{relsize}

\definecolor{darkred}{RGB}{139,0,0}
\definecolor{darkgreen}{RGB}{0,100,0}
\definecolor{darkmagenta}{RGB}{170,0,120}
\definecolor{darkpurple}{RGB}{110,0,180}
\definecolor{darkblue}{RGB}{40,0,200}
\definecolor{darkbrown}{rgb}{0.75,0.40,0.15}

\begin{document}
\title{Tent-transformed lattice rules for integration and approximation of multivariate non-periodic functions}
\author{Ronald Cools\thanks{Department of Computer Science, KU Leuven, Celestijnenlaan 200A, 3001 Leuven, Belgium. \newline
   Email: ronald.cools@cs.kuleuven.be, dirk.nuyens@cs.kuleuven.be, gowri.suryanarayana@cs.kuleuven.be}
   \and Frances Y. Kuo\thanks{School of Mathematics and Statistics, University of New South Wales, Sydney NSW 2052,
   Australia. \newline
   Email: f.kuo@unsw.edu.au}
   \and Dirk Nuyens\footnotemark[1]
   \and Gowri Suryanarayana\footnotemark[1]
}
\maketitle
\begin{abstract}
We develop algorithms for multivariate integration and approximation in
the weighted half-period cosine space of smooth non-periodic functions.
We use specially constructed tent-transformed rank-$1$ lattice points as
cubature nodes for integration and as sampling points for approximation.
For both integration and approximation, we study the connection between
the worst-case errors of our algorithms in the cosine space and the
worst-case errors of some related algorithms in the well-known weighted
Korobov space of smooth periodic functions. By exploiting this connection,
we are able to obtain constructive worst-case error bounds with good
convergence rates for the cosine space.


\end{abstract}

\section{Introduction}

In this paper we consider multivariate integration and approximation in
the \emph{weighted half-period cosine space}. We use
\emph{tent-transformed rank-$1$ lattice points} as cubature nodes for
integration and as sampling points for approximation. Lattice rules have
been widely studied in the context of multivariate integration, see
\cite{CN2008,N_review,SJ1994}. Rank-$1$ lattice point sets are completely
described by the number of points $n$ and an integer \emph{generating
vector} $\bsz$, which can be constructed by an algorithm that searches for
its elements \textit{component by component}, see e.g.,
\cite{D03, Kuo2003, NC06b,NC06a, SKJ02b, SKJ02a,SR02}.

We will focus on the non-periodic setting and, as in \cite{DNP2012}, we
will use the half-period cosine space spanned by the cosine series. Cosine
series are used for the expansion of non-periodic functions in the
$d$-dimensional unit cube. They are the eigenfunctions of the Laplace
differential operator with homogeneous Neumann boundary conditions. The
half-period cosine functions form a set of orthonormal basis functions of
$L_2([0,1])$ and are given by
\[
 \phi_0(x) \,=\, 1, \qquad\mbox{and}\qquad
 \phi_k(x) \,=\, \sqrt{2}\cos(\pi kx ) \quad \mbox{ for } k \in \N.
\]
In $d$ dimensions we will use the tensor products of these functions
\begin{align} \label{cosinefunctions}
 \phi_{\bsk}(\bsx) \,:=\, \prod _{j = 1}^{d}\phi_{k_j}(x_j)
 \,=\, \sqrt{2}^{|\bsk|_0}\prod _{j = 1}^{d}\cos(\pi k_j x_j),
\end{align}
where we denote by ${|\bsk|_0}$ the number of non-zero elements of $\bsk
\in \mathbb{Z}_+^d$, with
\[
  \mathbb{Z}_+ \,:=\, \{0,1,2,\ldots\}.
\]
The cosine series expansion of a $d$-variate function $f \in L_2([0,1]^d)$
converges to $f$ in the $L_2$ norm. Additionally, if $f$ is continuously
differentiable, we have uniform convergence and $f$ can be
expressed as a cosine series expansion as follows, see \cite{adcock,RA1}:
\begin{equation*}
 f(\bsx) \,=\, \sum _{\bsk \in \mathbb{Z}_+^d}  \hat{f}(\bsk)\, \phi_{\bsk}(\bsx),
\end{equation*}
where $\hat{f}(\bsk)$ are the cosine coefficients of $f$ and are obtained as
follows
\[
 \hat{f}(\bsk) \,=\, \int _{[0, 1]^d} f(\bsx)\, \phi_{\bsk}(\bsx) \rd\bsx.
\]
Cosine series overcome the well-known Gibbs phenomenon, which traditional
Fourier series face in the expansion of non-periodic functions. Cosine
series and the spectral methods using them have been studied in depth in
\cite{adcock,RA1} and their successors.

The precise definition of the weighted half-period cosine space will be
presented in Section~\ref{sec:prob}. For now we mention only that there is
a parameter $\alpha>1/2$ which characterizes the smoothness of the space
by controlling the decay of the cosine coefficients, and there is a
sequence of weights $1\ge \gamma_1\ge\gamma_2>\cdots > 0$ which models the
relative importance between successive coordinate directions.

We will first look at the problem of multivariate integration, where we
will use tent-transformed lattice points as cubature nodes. Lattice rules
have traditionally been used for the integration of smooth periodic
functions. In the Korobov space of smooth periodic functions, it is known
that lattice rules with well-chosen generating vectors can achieve the
(almost optimal) rate of convergence of $\calO(n^{-\alpha+\delta})$, for any
$\delta >0$, see, \eg, \cite{D03,Kuo2003}. Moreover, the result for the
case $\alpha=1$ can be used to prove that randomly-shifted lattice rules
can achieve the (almost optimal) rate of convergence of $\calO(n^{-1+\delta})$ for
$\delta >0$ in the Sobolev spaces of non-periodic functions of dominating
mixed smoothness $1$. Tent-transformed lattice rules were first used to
integrate non-periodic functions in \cite{Hick2002}, in the setting of
unanchored Sobolev spaces of dominating mixed smoothness $1$ and~$2$. It
was shown there that when the lattice points are first randomly shifted
and then tent-transformed (called bakers' transform in \cite{Hick2002}),
they can achieve the convergence rates of $\calO(n^{-1+\delta})$ and
$\calO(n^{-2+\delta})$, $\delta>0$, in the Sobolev spaces of smoothness
$1$ and~$2$, respectively.

In \cite{DNP2012}, tent-transformed lattice points were studied for
integration in the weighted half-period cosine space without random shifting. 
It was claimed there that the worst-case error in the cosine space for a
tent-transformed lattice rule is the same as the worst-case error in the
weighted Korobov space of smooth periodic functions using lattice rules,
given the same set of weights $\gamma_j$ and the smoothness parameter
$\alpha$. The argument was based on achieving equality in a Cauchy--Schwarz
type error bound, however the authors did not realise that this equality is not 
always possible in this setting. In this paper we correct this by showing 
that the worst-case error in the Korobov space is in fact an upper bound to 
the worst-case error in the cosine space and we provide an expression
for the scaling factor involved. We also conclude that, with an
appropriate rescaling of the weights $\gamma_j$, all the results for
integration in Korobov spaces using lattice rules, \eg, \cite{D03,
Kuo2003, NC06b}, also apply to integration in the cosine space using
tent-transformed randomly-shifted lattice rules (first randomly shifted
and then tent-transformed). Note additionally that the cosine space of
smoothness $1$ coincides with the unanchored Sobolev space of smoothness
$1$, see \cite{DNP2012}. Thus our results apply to the unanchored Sobolev
space of smoothness $1$ as well. 

The second part of our paper deals with the approximation of non-periodic
functions $f : [0,1]^d \to \mathbb{C}$ where the number of variables $d$
is large. Lattice rules have already been used for approximation in
weighted Korobov spaces, \eg, in \cite{KSW2006} and \cite{KWW09} in
the $L_2$ and $L_{\infty}$ settings, respectively. The use of lattice points 
for the approximation of periodic functions  was also suggested
much earlier in \cite{Kor63} and in papers cited there, see also
\cite{T2,T3,T1}. Lattice points were also used in
\cite{fred} for a spectral collocation method with Fourier basis where
samples of the input function at lattice points were used to approximate
the solution of PDEs such as the Poisson equation in $d$ dimensions. The
paper \cite{Lutz} also presented an approach for stably reconstructing
multivariate trigonometric polynomials (periodic) that have support on a
hyperbolic cross, by sampling them on rank-1 lattices. More advances
on the topic of reconstruction of trigonometric polynomials using rank-$1$ lattice sampling
 can be found, \eg, in \cite{Kam2014, KPV2015, PV2015}.

Our study is for non-periodic functions belonging to weighted cosine
spaces. In \cite{SNC2015}, collocation and reconstruction
problems were extended to non-periodic functions in the cosine space using
tent-transformed lattice point sets; however, that paper did not include
error analysis. Multivariate cosine expansions have also been studied
alongside hyperbolic cross approximations in \cite{adcock} and
\cite{AH_proc}. In \cite{adcock}, however, it was assumed that the error
in approximating the cosine coefficients is negligible. We fill this gap
by giving a detailed analysis of the error components. We first find the
expression of the error for the algorithm using $n$ function values at
tent-transformed lattice points for an arbitrary generating vector $\bsz$.
We then show that an upper bound for the worst-case error of our algorithm
in the cosine space using tent-transformed lattice points is the same as
an upper bound presented in \cite{KSW2006} for a related algorithm using
lattice points in the Korobov space. We can hence inherit all the error
bounds as well as the construction algorithms. In \cite{BKUT2015}, it is
shown that the convergence rate of rank-$1$ lattice points for function
approximation in the periodic Sobolev space of hybrid mixed smoothness
$\alpha$ is $\alpha/2$. This is only half of the optimal rate, achieved
for instance by sparse grid sampling. However, as mentioned in
\cite{BKUT2015}, rank-$1$ lattice point sets are still a convenient choice
for a number of reasons. The computations in higher dimensions can be
reduced to one-dimensional FFT and IFFT. Also, after applying the tent transformation,
which is computationally very inexpensive, these point sets become
suitable for the non-periodic setting immediately.

We now summarize the content of this paper. In Section~\ref{sec:prob} we
define the weighted cosine space and related function spaces, as well as
rank-$1$ lattice and tent-transformed rank-$1$ lattice point sets. In
Sections~\ref{sec:int} and~\ref{sec:app} we focus on the problems of
integration and approximation, respectively. In both sections, we derive
the worst-case errors for our algorithms based on tent-transformed lattice
point sets, and relate these errors to those of the Korobov space to
obtain results on the construction algorithms and convergence results.
Finally, Section~\ref{sec:conc} provides some concluding remarks.

\section{Problem setting} \label{sec:prob}

We want to integrate and approximate functions belonging to some
\emph{weighted $\alpha$-smooth half-period cosine space} (henceforth we
refer to it as the ``cosine space'' to be concise) of complex-valued
functions, given by
\begin{align*} 
 C_{d,\alpha,\bsgamma}
 \,:=\,
 \left\{ f\in L_2([0,1]^d) \;:\;
 \|f\|_{C_{d,\alpha,\bsgamma}}^2 := \sum_{\bsk \in \Z_+^d} |\hat{f}(\bsk)|^2\, r_{\alpha,\bm{\gamma}}(\bsk) < \infty
 \right\},
\end{align*}
where $\alpha > 1/2$ is a smoothness parameter and $\bsgamma
=(\gamma_1,\gamma_2,\ldots)$ is a sequence of \emph{weights} satisfying $1
\geq \gamma_1\geq\gamma_2\geq\cdots
>0$, and where we define
\begin{align*}
 r_{\alpha,\bsgamma}(\bsk) \,:=\, \prod_{j = 1}^dr_{\alpha,\gamma_j}(k_j),
 \qquad\mbox{with}\qquad
 r_{\alpha,\gamma_j}(k)
 \,:=\, \begin{cases}
   1 &  \text{if $k= 0$},\\
   |k|^{2\alpha}/\gamma_j & \text{if $k\ne0$}.
   \end{cases}
\end{align*}
Here we assume that successive variables have diminishing importance, with
each weight $\gamma_j$ moderating the behavior of the $j$th variable. If
all $\gamma_j=1$, we have the unweighted space where all variables are
equally important. If, however, $\gamma_j$ is small then the dependence on
the $j$th variable is weak. The smoothness parameter
$\alpha$ controls the decay of spectral coefficients, measured in the
$L_2$ sense. For $\alpha>1/2$,
the cosine space is a reproducing
kernel Hilbert space, with the reproducing kernel
\begin{align} \label{eq:Kd}
 K_{d,\alpha,\bsgamma}(\bsx,\bsy)
 \,:=\, \sum_{\bsk\in\Z_+^d} \frac{\phi_\bsk(\bsx)\,\phi_\bsk(\bsy)}{r_{\alpha,\bsgamma}(\bsk)}
 &\,=\, \sum_{\bsk\in\Z_+^d} \frac{2^{|\bsk|_0}}{r_{\alpha,\bsgamma}(\bsk)} \prod_{j=1}^d \cos(\pi k_jx_j)\cos(\pi k_jy_j) \\
 &\,=\, \prod_{j=1}^d \left(1 + 2\gamma_j \sum_{k=1}^\infty \frac{\cos(\pi kx_j)\cos(\pi ky_j)}{k^{2\alpha}} \right),
 \qquad\bsx,\bsy\in [0,1]^d. \nonumber
\end{align}
Recall that the reproducing kernel satisfies $K_d(\cdot,\bsy)\in C_d$ for
all $\bsy\in [0,1]^d$ as well as the reproducing property $\langle f,
K_d(\cdot,\bsy)\rangle_{C_d} = f(\bsy)$ for all $\bsy\in [0,1]^d$ and all
$f\in C_d$, where the inner product is defined by $\langle
f,g\rangle_{C_d} := \sum_{\bsk \in \Z_+^d} \hat{f}(\bsk)\,
\overline{\hat{g}(\bsk)}\,r_{\alpha,\bsgamma}(\bsk)$. (For brevity we have
omitted some parameters $\alpha$ and $\bsgamma$ from our notation in the
discussion above.)

When $\alpha = 1$, it is proved in~\cite{DNP2012} that the cosine space
\textbf{coincides} with the \emph{unanchored Sobolev space of dominated
mixed smoothness~$1$}. For this space the norm is given by
\begin{align*}
  \|f\|_{C_{d,1,\bsgamma}}^2
  \,=\,
  \sum_{\setu \subseteq \{1,\ldots,d\}}
  \prod_{j\in\setu} \gamma_j^{-1}
  \int_{[0,1]^{|\setu|}}
  \left|
  \int_{[0,1]^{d-|\setu|}} \frac{\partial^{|\setu|}}{\partial \bsx_{\setu}}f(\bsx) \rd\bsx_{\{1,\ldots,d\}\setminus\setu} \right|^2
  \rd\bsx_{\setu},
\end{align*}
where $\bsx_\setu = (x_j)_{j\in\setu}$ and $\partial^{|\setu|}f/\partial
\bsx_{\setu}$ denotes the mixed first derivatives of $f$ with respect to
the variables $x_j$ with $j\in\setu$, and the reproducing kernel is
\begin{align*}
 K_{d,1,\bsgamma}(\bsx,\bsy)
 \,=\, \prod_{j=1}^d \left(1 + \gamma_jB_1(x_j)B_1(y_j)+\gamma_j \frac{B_2(|x_j-y_j|)}{2}\right),
 \qquad\bsx,\bsy\in [0,1]^d,
\end{align*}
where $B_1(x) = x-1/2$ and $B_2(x) = x^2-x + 1/6$ are the Bernoulli polynomials of degrees~$1$ and $2$ respectively.

Another function space closely related to the cosine space is the
\emph{weighted Korobov space} of periodic functions defined by
\begin{align*}
 E_{d,\alpha,\bsgamma}
 \,:=\,
 \left\{ f \in L_2([0,1]^d) \;:\;
 \|f\|_{E_{d,\alpha,\bsgamma}}^2
 := \sum_{\bsh \in \Z^d} |\widetilde{f}(\bsh)|^2 \, r_{\alpha,\bsgamma}(\bsh) < \infty \right\},
\end{align*}
which, instead of the cosine coefficients, makes use of the Fourier
coefficients of $f$ given by
\begin{align*}
  \widetilde{f}(\bsh)
  \,:=\,
  \int_{[0,1]^d} f(\bsx) \, \exp(-\twopii \bsh \cdot \bsx) \rd\bsx
  \qquad\mbox{for}\quad \bsh \in \Z^d.
\end{align*}
(Note that the cosine coefficients are marked with a hat and the Fourier
coefficients are marked with a tilde.) Here the smoothness parameter
$\alpha> 1/2$ and the weights $\bsgamma = (\gamma_1,\gamma_2,\ldots)$ have
analogous interpretations as in the cosine space. The reproducing kernel
is
\begin{align} \label{eq:Kper}
 K^{\rm{per}}_{d,\alpha,\bsgamma}(\bsx,\bsy)
 := \sum_{\bsh\in\Z^d}\frac{\exp(\twopii \bsh\cdot(\bsx-\bsy))}{r_{\alpha,\bm{\gamma}}(\bsh)}
 = \prod_{j=1}^d\left(1+2\gamma_j\sum_{h=1}^{\infty}\frac{\cos(2\pi h(x_j-y_j))}{h^{2\alpha}}\right)
 \quad\bsx,\bsy\in [0,1]^d.
\end{align}
We remark that in many earlier papers the definition of the Korobov space
has $2\alpha$ instead of $\alpha$ as the smoothness parameter, and
therefore care must be taken when quoting results from these papers.

In this paper we study multivariate integration and approximation in the
cosine space using ``tent-transformed lattice rules". For a given $n \in
\mathbb{N}$ and $\bsz \in \mathbb{Z}_n^d$ where $\mathbb{Z}_n =
\{0,1,\ldots,n-1\}$, a \emph{rank-$1$ lattice point set} $\Lambda(\bsz,
n)$ is given by
\begin{equation} \label{eq:lat}
 \Lambda(\bsz, n) \,:=\,
 \left\{\frac{i\bsz}{n} \bmod 1 \;:\; i = 1,2,\ldots,n\right\},
\end{equation}
and $\bsz$ is called the \emph{generating vector}. The \emph{tent
transformation} $\psi: [0,1] \rightarrow [0,1]$, is given by
\begin{equation} \label{tent}
 \psi(x) \,:=\, 1 - |2x -1|, \qquad x\in [0,1],
\end{equation}
and we write $\psi(\bsx):=(\psi(x_1),\psi(x_2),\ldots,\psi(x_d))$ to
denote a tent-transformed point $\bsx \in [0,1]^d$, where the
transformation $\psi$ is applied component-wise to all coordinates in
$\bsx$. We obtain a \emph{tent-transformed point multiset}
$\Lambda_{\psi}(\bsz, n)$ by applying the tent transformation
component-wise to all the points of the rank-$1$ point set $\Lambda(\bsz,
n)$, that is,
\begin{equation} \label{eq:lat-psi}
 \Lambda_{\psi}(\bsz, n) \,:=\,
 \left\{\psi\left(\frac{i\bsz}{n} \bmod 1\right) \;:\; i = 1,2,\ldots,n\right\}.
\end{equation}

We may also consider a \emph{shifted point set}, and a
\emph{tent-transformed shifted point multiset} (the points are first
shifted and then tent-transformed), that is, given a \emph{shift}
$\bsDelta\in [0,1]^d$ we define
\begin{align*}
 \Lambda(\bsz, \bsDelta,n) &\,:=\,
 \left\{\left(\frac{i\bsz}{n}+\bsDelta\right) \bmod 1 \;:\; i = 1,2,\ldots,n\right\}, \\
 \Lambda_{\psi}(\bsz, \bsDelta,n) &\,:=\,
 \left\{\psi\left(\left(\frac{i\bsz}{n}+\bsDelta\right) \bmod 1\right) \;:\; i = 1,2,\ldots,n\right\}.
\end{align*}

In the forthcoming sections, we will provide more details specific to the
problems of integration and approximation.

\section{Integration} \label{sec:int}

We first study multivariate integration
\begin{gather}\label{Int}
 \mathrm{INT}_d (f) \,:=\, \int_{[0,1]^d} f(\bsx) \rd \bsx
\end{gather}
for functions $f$ from the cosine space $C_{d,\alpha,\bsgamma}$. We will
approximate the integral \link{Int} by some weighted cubature rule
\begin{gather*}
 Q_{n}(f)
 \,:=\, \sum_{i=1}^n w_i\, f(\bm{t}_i),
\end{gather*}
where $d,n\in\N$, $\bm{t}_1,\ldots,\bm{t}_n\in[0,1]^d$ are the sampling
points, and $w_1,\ldots,w_n\in\bbR$ are the cubature weights.

A \emph{lattice rule} is a cubature rule which uses points from a lattice
$\Lambda(\bsz, n)$, see \eqref{eq:lat}, with equal weights $w_i=1/n$, and
we will denote its application to a function $f$ by $Q_{n}(f; \bsz)$.
Likewise, a \emph{tent-transformed lattice rule} uses points from the
tent-transformed point multiset $\Lambda_{\psi}(\bsz, n)$, see
\eqref{eq:lat-psi}, again with equal weights $1/n$, and we will denote it
by $Q_{n}(f\circ\psi; \bsz)$. Note that transforming the input argument to
a function is equivalent to transforming the function itself, i.e.,
$f(\psi(\bsx)) = (f\circ \psi)(\bsx)$, hence our 
notation $Q_{n}(f\circ\psi; \bsz)$.

Analogously, we denote a \emph{shifted lattice rule} by
$Q_n(f;\bsz,\bsDelta)$, and a \emph{tent-transformed shifted lattice rule}
by $Q_n(f\circ\psi; \bsz,\bsDelta)$. If the shift $\bsDelta$ is generated
randomly from the uniform distribution on $[0,1]^d$, then we denote the
corresponding randomized methods by $Q_n^{\rm ran}(f;\bsz)$ and $Q_n^{\rm
ran}(f\circ\psi; \bsz)$, respectively.

The set of indices for those Fourier frequencies that are not integrated
exactly by the lattice rule, together with the index $\bszero$, is called
the \emph{dual} of the lattice and is given by
\[
 \Lambda(\bsz, n)^{\perp} \,:=\, \left\{\bsh\in\Z^d \;:\; \bsh\cdot\bsz \equiv 0\pmod{n} \right\}.
\]
More precisely, from \cite[Lemma~5.21]{N92}, we have
\begin{align} \label{identity}
\frac{1}{n}\sum_{\bst \in \Lambda(\bsz, n)}\exp(\twopii
\bsh\cdot\bst) \,=\, \begin{cases}
                          1 &  \text{ if $\bsh\in \Lambda(\bsz, n)^{\perp}$},\\%
			   0 & \text{ otherwise.}
                         \end{cases}
\end{align}
We will make use of this property in our analysis below.

In general, if $K$ is the reproducing kernel of some reproducing kernel
Hilbert space $H_d$ of functions on $[0,1]^d$, then the squared worst-case
error of $Q_{n}$ is given by (see, e.g., \cite{IntandApp})
\begin{align}
  &e^{\mathrm{wor}}(Q_{n}; H_d)^2
  \,:=\, \left( \sup_{f\in H_d,\, \|f\|_{H_d}\le 1} \abs{\mathrm{INT}_d (f) - Q_{n}(f)} \right)^2  \nonumber \\
  &\,=\, \int_{[0,1]^{2d}} K(\bsx,\bsy) \rd\bsx \rd\bsy - 2\sum_{i=1}^n w_i \int_{[0,1]^d} K(\bsx,\bm{t}_i) \rd\bsx
  + \sum_{i,i'=1}^n w_i w_{i'} K(\bm{t}_i,\bm{t}_{i'}).
 \label{eq:e2-K}
\end{align}
In Subsections~\ref{sec:tent} and~\ref{sec:tent-sh} below, we will make
use of this formula to derive and analyse the worst-case error for a
tent-transformed lattice rule and the root-mean-squared worst-case error
for a tent-transformed randomly-shifted lattice rule. 
We further need the following lemma and an identity following the lemma.

\begin{lemma} \label{cosexp}
Let $\psi(\bsx)$ be the tent transform function as in \eqref{tent}. For
any $\bsk \in \mathbb{Z}_+^d$ and the corresponding basis function
$\phi_{\bsk}$ as in \eqref{cosinefunctions}, we have
\begin{align} \label{char}
  \phi_{\bsk}(\psi(\bsx))
  \,=\, (\sqrt{2})^{ |\bsk|_0 } \prod _{j = 1}^{d}\cos(2\pi k_j x_j)
 \,=\, \frac{(\sqrt{2})^{|\bsk|_0 }}{2^{d}} \sum _{\bssigma \in \{\pm 1\}^d}\exp(\twopii \bssigma(\bsk)\cdot\bsx),
\end{align}
where $\bssigma \in\{\pm 1\}^d $ are sign combinations and
$\bssigma(\bsk)$ denotes the application of these signs on the indices of
$\bsk$ element-wise, \ie, $\bssigma(\bsk) = (\sigma_1k_1,\ldots,\sigma_dk_d)$, and as before $|\bsk|_0$ denotes the number of
non-zero elements in $\bsk$.
\end{lemma}

\begin{proof}
It is trivial to verify that for $k \in \Z_+$ we have $\cos(\pi k\,
\psi(x)) = \cos(2\pi k x)$. This yields the first equality in
\eqref{char}. Next we write
\begin{align*}
  \prod _{j = 1}^{d}\cos(2\pi k_j x_j)
 \,=\,  \frac{1}{2^d} \prod _{j = 1}^{d}(\exp(\twopii k_j x_j )+ \exp(-\twopii k_j x_j )).
\end{align*}
Expanding the product then yields the second equality in \eqref{char}.
\end{proof}

We will repeatedly use the following identity: for any two functions
$G_1,G_2:\Z^d\to\C$,
\begin{align} \label{identity_2}
 \sum_{\bsk \in \Z_+^d} \bigg(G_1(\bsk) \sum_{\bssigma \in \{\pm 1\}^d} G_2(\bssigma(\bsk))\bigg)
 \,=\, \sum_{\bsk \in \Z^d} G_1(|\bsk|)\, G_2(\bsk)\,2^{d-|\bsk|_0},
\end{align}
where $|\bsk|$ indicates that the absolute value function is applied
component-wise to the vector.

\subsection{Lower bound}

A lower bound for the worst-case error for integration in the cosine space
is known from \cite{DNP2012}, and is given in the following theorem.
\begin{theorem} \label{lb}
For arbitrary points $\bst_1,\ldots,\bst_n\in [0,1]^d$ and weights
$w_1,\ldots,w_n\in\mathbb{R}$, we have
\begin{align*}
 e^{\mathrm{wor}}(Q_{n}; C_{d,\alpha,\bsgamma}) \,\ge\, c(d,\alpha,\bsgamma)\,\frac{(\log n)^{(d-1)/2}}{n^{\alpha}},
\end{align*}
where $c(d,\alpha,\bsgamma) > 0$ depends on $d$, $\alpha$, and $\bsgamma$,
but not on $n$, the points $\bst_1,\ldots,\bst_n$, or the weights $w_1,\ldots,w_n$.
\end{theorem}

\subsection{Upper bound for tent-transformed lattice rules} \label{sec:tent}

The following theorem gives the formula for the worst-case integration
error for a tent-transformed lattice rule in the cosine space.

\begin{theorem}
The squared worst-case error for a tent-transformed lattice rule in the
cosine space is given by
\begin{align} \label{eq:psi-err}
 e^{\mathrm{wor}}(Q_{n}(\cdot\circ\psi;\bsz);C_{d,\alpha,\bsgamma})^2
 &\,=\, \sum_{\bszero\ne \bsk \in \Lambda(\bsz,n)^{\perp}}\frac{1}{r_{\alpha,\bm{\gamma}}(\bsk)}
 \left(\frac{1}{2^{d}}\sum _{\bssigma\in \{\pm 1\}^d}\ind{\bssigma(\bsk)\in \Lambda(\bsz,n)^{\perp}}\right).
\end{align}
\end{theorem}

\begin{proof}
Using \eqref{eq:e2-K} and \eqref{eq:Kd}, and then applying \eqref{char},
\eqref{identity}, \eqref{identity_2} in turn, we obtain
\begin{align*}
 &e^{\mathrm{wor}}(Q_{n}(\cdot\circ\psi;\bsz);C_{d,\alpha,\bsgamma})^2
 \,=\, -1+ \frac{1}{n^2}\sum_{\bst,\bst' \in \Lambda(\bsz, n)}\sum_{\bsk \in \Z_+^d}
  \frac{\phi_k(\psi(\bst))\,\phi_k(\psi(\bst'))}{r_{\alpha,\bm{\gamma}}(\bsk)} \nonumber\\
 &\,=\, -1 + \frac{1}{n^2} \sum_{\bst,\bst' \in \Lambda(\bsz, n)}\sum_{\bsk \in \Z_+^d} \frac{2^{|\bsk|_0}}{r_{\alpha,\bm{\gamma}}(\bsk)}
 \left(\frac{1}{2^{d}}\sum _{\bssigma\in \{\pm 1\}^d}\exp(\twopii\bssigma(\bsk)\cdot\bst)\right) \\
 &\qquad\qquad\qquad\qquad\qquad\qquad \times
 \left(\frac{1}{2^{d}}\sum _{\bssigma'\in \{\pm 1\}^d}\exp(\twopii\bssigma'(\bsk)\cdot\bst')\right)\nonumber\\
 &\,=\, -1 +\sum_{\bsk \in \Z_+^d} \frac{2^{|\bsk|_0}}{r_{\alpha,\bm{\gamma}}(\bsk)}
 \left(\frac{1}{2^{d}}\sum _{\bssigma\in \{\pm 1\}^d}\ind{\bssigma(\bsk)\in\Lambda(\bsz,n)^{\perp}}\right)
 \left(\frac{1}{2^{d}}\sum _{\bssigma'\in \{\pm 1\}^d}\ind{\bssigma'(\bsk)\in\Lambda(\bsz,n)^{\perp}}\right)\nonumber\\
 &\,=\, -1 +\sum_{\bsk \in \Z^d}\frac{1}{ r_{\alpha,\bm{\gamma}}(|\bsk|)}\ind{\bsk\in\Lambda(\bsz,n)^{\perp}}
 \left(\frac{1}{2^{d}}\sum _{\bssigma\in \{\pm 1\}^d}\ind{\bssigma(|\bsk|)\in\Lambda(\bsz,n)^{\perp}}\right),
 \end{align*}
which yields \eqref{eq:psi-err}.
\end{proof}

In comparison, the squared worst-case error for a lattice rule in the
Korobov space is (see, e.g., \cite{D03, Kuo2003})
\begin{align} \label{korerror}
 e^{\mathrm{wor}}(Q_{n}(\cdot;\bsz);E_{d,\alpha,\bsgamma})^2
 &\,=\, \sum_{\bszero \ne \bsk \in \Lambda(\bsz,n)^{\perp}} \frac{1}{r_{\alpha,\bm{\gamma}}(\bsk)}.
\end{align}
Clearly \eqref{korerror} is an upper bound for \eqref{eq:psi-err}, since
the formula \eqref{eq:psi-err} involves an additional factor which is
always~$\le 1$. 
This was not recognized in \cite{DNP2012}. Nevertheless, it is true that one may
borrow the result from the Korobov space for the cosine space. We
formalize this conclusion in the corollary below. For simplicity we state
the result only for a prime $n$, but a similar result for general $n$ is also
known, see \cite{D03,Kuo2003, NC06b,NC06a}.

\begin{corollary}
A fast component-by-component algorithm can be used to obtain a generating
vector $\bsz\in\Z^d_n$ in $\calO(d\,n\log n)$ operations, using the
squared worst-case error for a lattice rule in the Korobov space
$E_{d,\alpha,\bsgamma}$ as the search criterion, such that the worst-case
error for the resulting tent-transformed lattice rule in the cosine space
$C_{d,\alpha,\bsgamma}$ satisfies
\[
  e^{\mathrm{wor}}(Q_{n}(\cdot\circ\psi;\bsz);C_{d,\alpha,\bsgamma})
  \,\le\, e^{\mathrm{wor}}(Q_{n}(\cdot;\bsz);E_{d,\alpha,\bsgamma})
  \,\le\, \left(\frac{1}{n-1}
  \left(\prod_{j=1}^d (1 + 2\zeta(2\alpha\lambda)\gamma_j^\lambda) - 1\right) \right)^{1/(2\lambda)}
\]
for all $1/(2\alpha) < \lambda\le 1$, where $\zeta(x) = \sum_{k=1}^\infty
k^{-x}$ is the Riemann zeta function. Hence, the convergence rate is
$\calO(n^{-1/(2\lambda)})$, with the implied constant independent of~$d$
if $\sum_{j=1}^\infty \gamma_j^\lambda < \infty$. As $\lambda\to
1/(2\alpha)$, the method achieves the optimal rate of convergence close to
$\calO(n^{-\alpha})$.
\end{corollary}

Ideally we would like to be able to perform the fast
component-by-component algorithm using the formula \eqref{eq:psi-err} as
the search criterion directly, rather than using its upper bound
\eqref{korerror}. However, to do this we must identify a strategy to
 handle the evaluation of the sum over all sign changes which
is of order $2^d$. This is left for future research.

We end this subsection by providing another insight into why the error
\eqref{eq:psi-err} in the cosine space is smaller than the error
\eqref{korerror} in the Korobov space. From~\eqref{eq:e2-K}, we can derive that
the worst-case error of a tent-transformed lattice rule in the cosine
space is the same as the worst-case error of the lattice rule in the
\emph{tent-transformed cosine space}, which is a reproducing kernel
Hilbert space with kernel
\begin{align} \label{eq:Kpsi}
  K^{\psi}_{d,\alpha,\bsgamma}(\bsx,\bsy)
  \,:=\, K_{d,\alpha,\bsgamma}(\psi(\bsx),\psi(\bsy))
  \,=\, \sum_{\bsk\in\Z_+^d} \frac{2^{|\bsk|_0}}{r_{\alpha,\bsgamma}(\bsk)}
  \prod_{j=1}^d \cos(2\pi k_j x_j)\,\cos(2\pi k_j y_j).
\end{align}
Indeed, we have
\begin{align*}
  e^{\mathrm{wor}}(Q_{n}(\cdot\circ\psi;\bsz);C_{d,\alpha,\bsgamma})^2
  &\,=\, -1+ \frac{1}{n^2}\sum_{\bst,\bst' \in \Lambda_{\psi}(\bsz, n)}K_{d,\alpha,\bsgamma}(\bst,\bst') \\
  &\,=\, -1+\frac{1}{n^2}\sum_{\bst,\bst' \in \Lambda(\bsz, n)}K_{d,\alpha,\bsgamma}(\psi(\bst),\psi(\bst')).
\end{align*}
It can be shown that the kernel of the tent-transformed cosine space
is smaller than the kernel of the Korobov space, \ie,
$K_{d,\alpha,\bsgamma}^{\rm{per}}(\bsx,\bsy) -
K_{d,\alpha,\bsgamma}^{\psi}(\bsx,\bsy)$ is positive definite. From the
theory of reproducing kernels \cite{Aronszajn}, we then know that the
tent-transformed cosine space is a subspace of the Korobov space,
and hence the worst-case error of a lattice rule in the tent-transformed
cosine space is at most its worst-case error in the Korobov space.

\subsection{Upper bound for tent-transformed randomly-shifted lattice
rules} \label{sec:tent-sh}

We now consider the randomized method $Q_n^{\rm ran}(f\circ\psi;\bsz)$.
Recall that in a tent-transformed shifted lattice rule
$Q_n(f\circ\psi;\bsz,\bsDelta)$ we first shift the lattice point set and
then apply the tent transformation. In the randomized method the shift
$\bsDelta$ is generated randomly from the uniform distribution on
$[0,1]^d$. To show the existence of good shifts $\bsDelta$, we analyze the
\emph{root-mean-squared worst-case error} defined by
\begin{equation*}
 e^{\mathrm{wor}}_{\mathrm{rms}}(Q_n^{\rm ran}(\cdot\circ\psi;\bsz); C_{d,\alpha,\bsgamma})
 \,:=\, \left( \int_{[0,1]^d} e^{\mathrm{wor}}(Q_n(\cdot\circ\psi;\bsz,\bsDelta); C_{d,\alpha,\bsgamma})^2
 \rd\bsDelta \right)^{1/2}.
\end{equation*}
{}From \cite{Hick2002} we know that
\begin{align} \label{eq:rms-err}
 e^{\mathrm{wor}}_{\mathrm{rms}}(Q_n^{\rm ran}(\cdot\circ\psi;\bsz); C_{d,\alpha,\bsgamma})^2
 \,=\, -1+ \frac{1}{n^2}\sum_{\bst,\bst' \in \Lambda(\bsz, n)} K_{d,\alpha,\bsgamma}^{\sh,{\psi}}(\bst,\bst'),
\end{align}
where $K_{d,\alpha,\bsgamma}^{\sh,{\psi}}$ is the \emph{shift-invariant
tent-transformed kernel associated with} $K_{d,\alpha,\bsgamma}$, given by
\begin{equation}\label{def_Kshinv}
  K_{d,\alpha,\bsgamma}^{\sh,\psi}(\bsx, \bsy)
  \,:=\, \int_{[0,1)^d} K_{d,\alpha,\bsgamma}(\psi(\bsx+\bsDelta),\psi(\bsy+\bsDelta)) \rd\bsDelta,
  \qquad \bsx,\bsy\in[0,1]^d.
\end{equation}

\begin{theorem}\label{prop:Kshinv}
The shift-invariant tent-transformed kernel defined in \eqref{def_Kshinv}
can be written as
\begin{equation} \label{eq:Ksh-psi}
 K_{d,\alpha,\bsgamma}^{\sh,{\psi}}(\bsx,\bsy)
 \,=\, \sum_{\bsk\in\Z^d}\frac{2^{-|\bsk|_0}}{r_{\alpha,\bsgamma}(\bsk)}\exp(\twopii\bsk\cdot (\bsx-\bsy))
 \,=\, K_{d,\alpha,\bsgamma/2}^{\rm per}(\bsx,\bsy),
 \qquad \bsx,\bsy\in[0,1]^d.
\end{equation}
That is, it is precisely the kernel for the Korobov space with weights
$\bsgamma$ replaced by $\bsgamma/2$.
\end{theorem}

\begin{proof}
Starting from \eqref{def_Kshinv} and \eqref{eq:Kpsi}, we have
\begin{align*}
 K_{d,\alpha,\bsgamma}^{\sh,{\psi}}(\bsx,\bsy)
 &\,=\, \int_{[0,1]^d}\sum_{\bsk\in\Z_+^d}\frac{2^{|\bsk|_0}}{r_{\alpha,\bsgamma}(\bsk)}
 \prod_{j=1}^{d}\cos(2\pi k_j(x_j+\Delta_j))\cos(2\pi k_j(y_j+\Delta_j))\rd\bsDelta \\
 &\,=\, \sum_{\bsk\in\Z_+^d}\frac{2^{|\bsk|_0}}{r_{\alpha,\bsgamma}(\bsk)}
 \prod_{j=1}^{d}\left(\int_0^1\cos(2\pi k_j(x_j+\Delta_j))\cos(2\pi k_j(y_j+\Delta_j))\rd\Delta_j\right) \\
 &\,=\, \sum_{\bsk\in\Z_+^d}\frac{1}{r_{\alpha,\bsgamma}(\bsk)}\prod_{j=1}^{d}\cos(2\pi k_j(x_j-y_j)) \\
 &\,=\, \sum_{\bsk\in\Z_+^d}\frac{1}{r_{\alpha,\bsgamma}(\bsk)}
 \left(\frac{1}{2^{d}}\sum _{\bssigma\in \{\pm 1\}^d}\exp(\twopii\bssigma(\bsk)\cdot (\bsx-\bsy))
	    \right).
\end{align*}
Applying the identity \eqref{identity_2} then yields the first equality in
\eqref{eq:Ksh-psi}. The second equality in \eqref{eq:Ksh-psi} follows
immediately by a comparison with the formula \eqref{eq:Kper}, noting that
$2^{|\bsk|_0}\, r_{\alpha,\bsgamma}(\bsk) = r_{\alpha,\bsgamma/2}(\bsk)$
for all $\bsk\in\Z^d$.
\end{proof}

\begin{theorem}
The root-mean-squared worst-case error for a tent-transformed
randomly-shifted lattice rule in the cosine space is given by
\begin{align} \label{eq:rms-err-exp}
 e^{\mathrm{wor}}_{\mathrm{rms}}(Q_n^{\rm ran}(\cdot\circ\psi;\bsz); C_{d,\alpha,\bsgamma})^2
 \,=\, \sum_{\bszero\ne\bsk\in\Lambda(\bsz,n)^{\perp}}\frac{2^{-|\bsk|_0}}{r_{\alpha,\bsgamma}(\bsk)}
 \,=\, e^{\mathrm{wor}}(Q_n(\cdot;\bsz); E_{d,\alpha,\bsgamma/2})^2.
\end{align}
That is, it is precisely the squared worst-case error of the lattice rule
in the Korobov space with weights $\bsgamma$ replaced by $\bsgamma/2$.
\end{theorem}

\begin{proof}
The first equality in \eqref{eq:rms-err-exp} follows by combining
\eqref{eq:rms-err} with \eqref{eq:Ksh-psi} and using \eqref{identity}. The
second equality in \eqref{eq:rms-err-exp} then follows immediately by
comparison with the formula~\eqref{korerror}, noting again that
$2^{|\bsk|_0}\, r_{\alpha,\bsgamma}(\bsk) = r_{\alpha,\bsgamma/2}(\bsk)$
for all $\bsk\in\Z^d$.
\end{proof}

Due to the precise connection with the Korobov space, we can again borrow
all results from the Korobov space for the cosine space as mentioned in
\cite{DNP2012}, but this time with all weights scaled by a factor of $2$.
We summarize this conclusion in the corollary below.

\begin{corollary}
Let $\bsz\in\Z^d_n$ be the generating
vector obtained by a fast component-by-component algorithm in $\calO(d\,n\log n)$ operations, using the
squared worst-case error for a lattice rule in the Korobov space
$E_{d,\alpha,\bsgamma/2}$ as the search criterion. Then there exists a shift $\bsDelta\in [0,1]^d$ such that the
worst-case error for the resulting tent-transformed shifted lattice rule with the generating vector $\bsz$
in the cosine space $C_{d,\alpha,\bsgamma}$  satisfies
\begin{align*}
  e^{\mathrm{wor}}(Q_{n}(\cdot\circ\psi;\bsz,\bsDelta);C_{d,\alpha,\bsgamma})
  &\,\le\, e^{\mathrm{wor}}_{\mathrm{rms}}(Q_n^{\rm ran}(\cdot\circ\psi;\bsz); C_{d,\alpha,\bsgamma})
  \,=\, e^{\mathrm{wor}}(Q_{n}(\cdot;\bsz);E_{d,\alpha,\bsgamma/2}) \\
  &\,\le\, \left(\frac{1}{n-1}
  \left(\prod_{j=1}^d (1 + 2^{1-\lambda}\, \zeta(2\alpha\lambda)\gamma_j^\lambda) - 1\right) \right)^{1/(2\lambda)}
\end{align*}
for all $1/(2\alpha) < \lambda\le 1$, where $\zeta(\cdot)$ is again the
Riemann zeta function. Hence, the convergence rate is
$\calO(n^{-1/(2\lambda)})$, with the implied constant independent of~$d$
if $\sum_{j=1}^\infty \gamma_j^\lambda < \infty$. As $\lambda\to
1/(2\alpha)$, the method achieves the optimal rate of convergence close to
$\calO(n^{-\alpha})$.
\end{corollary}

\section{Function Approximation} \label{sec:app}

We define approximation in terms of the operator which is the embedding
from the cosine space to the $L_2$ space, \ie,
$\mathrm{APP}_d:C_{d,\alpha,\bsgamma} \to L_2([0,1]^d)$, and
\[
\mathrm{APP}_d(f) \,:=\, f.
\]
To approximate $\mathrm{APP}_d$, we study linear algorithms of the form
\begin{align}\label{linalg}
 A_{n,d}(f)(\bsx) \,=\, \sum_{i = 1}^n f(\bst_i)\,a_i(\bsx),
\end{align}
for some functions $a_i \in L_2([0,1]^d)$ and deterministically chosen
sample points $\bst_i \in [0,1]^d $. In particular, we are interested in
tent-transformed rank-$1$ lattice points for sampling.

For approximating the function $f$ from its samples, we consider a
hyperbolic cross index set for truncating the cosine series expansion. As
cosine series have spectral support only on the positive hyperoctant, we
define the \emph{weighted hyperbolic cross} $H_M$ on the positive
hyperoctant~by
\begin{align} \label{hypcross}
  H_M \,:=\, H_M^{d,\alpha,\bsgamma}
 \,:=\, \bigg\{\bsk \in \mathbb{Z}_+^d \;:\;  r_{\alpha,\bm{\gamma}}(\bsk)\leq M\bigg\},
\end{align}
with $M \in \mathbb{R}$ and $M \geq 1$. We approximate $f$ by first
truncating its cosine series expansion to $H_M$ and then approximating the
cosine coefficients for $\bsk \in H_M$ by an $n$-point tent-transformed
rank-$1$ lattice rule. So we have
\begin{align}
\label{approx_alg}
 A_{n,d,M}(f)(\bsx) \,:=\,
 \sum_{\bsk \in H_M} \bigg(\frac{1}{n}\sum_{\bst \in \Lambda_{\psi}(\bsz, n)}
 f(\bst)\,\phi_{\bsk}(\bst)\bigg)\phi_{\bsk}(\bsx).
\end{align}
That is, $f$ is approximated by a linear algorithm of the form~\eqref{linalg} with $\bst_i$ from $\Lambda_{\psi}(\bsz, n)$
and
\[
a_i(\bsx) = \frac{1}{n}\sum_{\bsk \in H_M}\phi_{\bsk}(\bst_i)\phi_{\bsk}(\bsx).
\]
We are then interested in the worst-case error of the algorithm $A_{n,d,M}$,
which is defined as follows
\[
  e^{\mathrm{wor}}(A_{n,d,M}; C_{d,\alpha,\bsgamma})
  \,:=\, \sup_{f\in C_{d,\alpha,\bsgamma},\, \|f\|_{C_{d,\alpha,\bsgamma}} \le 1} \|f - A_{n,d,M}(f)\|_{L_2([0,1]^d)}.
\]

\subsection{Upper bound on the worst-case error}

The following theorem gives the expression for the $L_2$ error of the
algorithm.

\begin{theorem}
The $L_2$ error of approximating $f$ by first truncating the spectral
expansion to a hyperbolic cross $H_M$ and then using a tent-transformed
rank-$1$ lattice rule with points
 $\Lambda_{\psi}(\bsz, n)$ to approximate the cosine coefficients is given by
\begin{align} \label{error}
 &\|f - A_{n,d,M}(f)\|^2_{L_2([0,1]^d)} \nonumber\\
 &\,=\, \sum_{\bsk \not\in H_M}|\hat{f}(\bsk)|^2 + \sum_{\bsk \in H_M}
 \bigg|\sum _{\bszero\ne\bsh \in \Lambda(\bsz,n)^{\perp}} \sum _{\bssigma \in \{\pm 1\}^d}
 \hat{f}(|\bssigma(\bsh) + \bsk|) \frac{(\sqrt{2})^{-|\bssigma(\bsh) + \bsk|_0 +|\bsk|_0}}{2^{d}}\bigg|^2,
\end{align}
where $\Lambda(\bsz,n)^{\perp}$ is the dual of $\Lambda(\bsz, n)$.
\end{theorem}

\begin{proof}
Clearly the approximation error of our algorithm \eqref{approx_alg} is
\begin{align*}
 (f-A_{n,d,M}(f))(\bsx)
 \,=\, \sum_{\bsk \not\in H_M}\hat{f}(\bsk)\,\phi_{\bsk}(\bsx)
 + \sum_{\bsk \in H_M}\left(\hat{f}(\bsk)-\hat{f}_a(\bsk)\right)
 \phi_{\bsk}(\bsx),
\end{align*}
where we denote by $\hat{f}_a(\bsk)$ the approximation of $\hat{f}(\bsk)$,
\ie,
\begin{align*}
 \hat{f}_a(\bsk)
 \,:=\, \frac{1}{n} \sum _{\bst \in \Lambda_{\psi}(\bsz, n)} f(\bst)\, \phi_{\bsk}(\bst).
\end{align*}
Since $\phi_{\bsk}$ is a set of orthonormal basis functions, we conclude
that
\begin{align} \label{error_exp}
 \|f - A_{n,d,M}(f)\|^2_{L_2([0,1]^d)}
 \,=\, \sum_{\bsk \not\in H_M} |\hat{f}(\bsk)|^2
 + \sum_{\bsk \in H_M} |\hat{f}(\bsk)-\hat{f}_a(\bsk)|^2.
\end{align}
To complete the proof we need to derive an explicit expression for
$\hat{f}_a(\bsk)$.

We can write
\begin{align*}
 \hat{f}_a(\bsk)
 \,=\, \frac{1}{n} \sum _{\bst \in \Lambda(\bsz, n)} f(\psi(\bst))\, \phi_{\bsk}(\psi(\bst))
 &\,=\, \frac{1}{n} \sum _{\bst \in \Lambda(\bsz, n)} \bigg(\sum _{\bsl \in \Z_+^d} \hat{f}(\bsl)\,
 \phi_{\bsell}(\psi(\bst))\bigg) \phi_{\bsk}(\psi(\bst)) \\
 &\,=\, \sum _{\bsl \in \Z_+^d} \hat{f}(\bsl)\,\frac{1}{n} \sum _{\bst \in \Lambda(\bsz, n)}
 \phi_{\bsell}(\psi(\bst))\, \phi_{\bsk}(\psi(\bst)).
\end{align*}
Using Lemma~\ref{cosexp} and
\eqref{identity}, we obtain
\begin{align*}
 \frac{1}{n}\sum_{\bst \in \Lambda(\bsz, n)}\phi_{\bsell}(\psi(\bst))\phi_{\bsk}(\psi(\bst))
 &\,=\, \frac{1}{n}\sum_{\bst \in \Lambda(\bsz, n)}  \frac{(\sqrt{2})^{|\bsl|_0 + |\bsk|_0 }}{2^{2d}}
		\sum _{\bssigma , \bssigma ' \in \{\pm 1\}^d}\exp(\twopii (\bssigma(\bsl) - \bssigma ' (\bsk))\cdot\bst) \\
 &\,=\, \frac{(\sqrt{2})^{|\bsl|_0 + |\bsk|_0 }}{2^{2d}}\sum _{\bssigma , \bssigma ' \in \{\pm 1\}^d}
\ind{\bssigma(\bsl) - \bssigma ' (\bsk) \in \Lambda(\bsz,n)^{\perp}}.
\end{align*}
Note that for any function $G:\Z^d\to\C$ and any
$\bsl,\bsk\in\Z_+^d$,
\begin{align*}
  \sum_{\bssigma,\bssigma' \in \{\pm 1\}^d} G(\bssigma(\bsl)-\bssigma'(\bsk))
  &\,=\, \sum_{\bssigma'\in \{\pm 1\}^d} \sum_{\bssigma \in \{\pm 1\}^d}
         G(\bssigma'({\bssigma'}^{-1}(\bssigma(\bsl))-\bsk)) \\
  &\,=\, \sum_{\bssigma'\in \{\pm 1\}^d} \sum_{\bssigma'' \in \{\pm 1\}^d}
         G(\bssigma'(\bssigma''(\bsl)-\bsk)).
\end{align*}
Here $\bssigma^{-1}$ is such that for any $\bsk\in\Z^d$, $(\bssigma^{-1}\circ\bssigma)(\bsk)=\bsk$. We thus arrive at
\begin{align*}
 \hat{f}_a(\bsk)
 &\,=\, \sum _{\bsl \in \Z_+^d} \hat{f}(\bsl) \frac{(\sqrt{2})^{|\bsl|_0 + |\bsk|_0 }}{2^{2d}}
 \sum _{\bssigma' , \bssigma '' \in \{\pm 1\}^d}\ind{\bssigma' (\bssigma''(\bsl) -  \bsk) \in \Lambda(\bsz,n)^{\perp}}\\
 &\,=\, \sum _{\bsl \in \Z^d}   \hat{f}(|\bsl|) \frac{(\sqrt{2})^{-|\bsl|_0 +|\bsk|_0 }}{2^{d}}
 \sum _{\bssigma'  \in \{\pm 1\}^d}\ind{\bssigma' (\bsl -  \bsk) \in \Lambda(\bsz,n)^{\perp}},
\end{align*}
where we applied \eqref{identity_2} to change the index set from $\Z_+^d$ to
$ \Z^d$. Taking $\bsl -\bsk = \bsh$ so that $\bsl = \bsh + \bsk$, we get
\begin{align*}
 \hat{f}_a(\bsk)
 &\,=\, \sum _{\bsh \in \Z^d}   \hat{f}(|\bsh + \bsk|) \frac{(\sqrt{2})^{-|\bsh + \bsk|_0 +|\bsk|_0 }}{2^{d}}
 \sum _{\bssigma ' \in \{\pm 1\}^d}\ind{\bssigma '(\bsh) \in \Lambda(\bsz,n)^{\perp}}.
\end{align*}
Since we sum over all $\bsh\in\Z^d$ as well as over all sign combinations
$\bssigma ' \in \{\pm 1\}^d$, the above expression can be regrouped as
\begin{align*}
 \hat{f}_a(\bsk)
 &\,=\, \sum _{\bsh \in \Z^d}  \ind{\bsh \in \Lambda(\bsz,n)^{\perp}} \sum _{\bssigma \in \{\pm 1\}^d} \hat{f}(|\bssigma(\bsh) + \bsk|) \frac{(\sqrt{2})^{-|\bssigma(\bsh) + \bsk|_0 +|\bsk|_0 }}{2^{d}}\\
 &\,=\, \hat{f}(\bsk) +
 \sum _{\bszero \ne \bsh \in \Lambda(\bsz,n)^{\perp}} \sum _{\bssigma \in \{\pm 1\}^d}
 \hat{f}(|\bssigma(\bsh) + \bsk|) \frac{(\sqrt{2})^{-|\bssigma(\bsh) + \bsk|_0 +|\bsk|_0 }}{2^{d}}.
\end{align*}
Substituting this into \eqref{error_exp} then completes the proof.
\end{proof}

The first term in the error expression \eqref{error} is the
\emph{truncation error} and the second term is the \emph{aliasing error}
that is accumulated from approximating the cosine coefficients using the
cubature rule. In the following theorem we estimate these two errors
separately to arrive at an upper bound on the squared worst-case error.

\begin{theorem}
The squared worst-case error for the algorithm $A_{n,d,M}$ which
approximates a function in $C_{d,\alpha,\bsgamma}$ by first truncating the
spectral expansion to a hyperbolic cross $H_M$ and then using a
tent-transformed rank-$1$ lattice rule with points $\Lambda_{\psi}(\bsz,
n)$ to evaluate the cosine coefficients is bounded by
\begin{align} \label{WCE_bound}
 e^{\mathrm{wor}}(A_{n,d,M}; C_{d,\alpha,\bsgamma})^2
 \,\le\, \frac{1}{M}
 + \sum_{\bsk \in H_M}\sum_{\bszero \ne \bsh \in \Lambda(\bsz,n)^{\perp}}\sum_{\bssigma \in \{\pm 1\}^d}
 \frac{2^{|\bsk|_0}}{2^d\,r_{\alpha,\bsgamma}(\bssigma(\bsh) + \bsk)}.
\end{align}
\end{theorem}

\begin{proof}
By the definition of $H_M$ in \eqref{hypcross} we have
$r_{\alpha,\bsgamma}(\bsk)> M$ for $\bsk\notin H_M$, and thus the truncation
error in \eqref{error} satisfies
\begin{align*}
 \sum_{\bsk \not\in H_M}|\hat{f}(\bsk)|^2
 &\,=\, \sum_{\bsk \not\in H_M}|\hat{f}(\bsk)|^2\, \frac{r_{\alpha,\bm{\gamma}}(\bsk)}{r_{\alpha,\bm{\gamma}}(\bsk)}
 \,<\, \frac{1}{M}\sum_{\bsk \in \Z_+^d}|\hat{f}(\bsk)|^2r_{\alpha,\bm{\gamma}}(\bsk)
 \,=\, \frac{1}{M} \|f\|^2_{C_{d,\alpha,\bsgamma}}.
\end{align*}
For the aliasing error in \eqref{error}, we first apply the Cauchy--Schwarz
inequality to obtain
\begin{align} \label{factors}
 &\bigg|\sum_{\bszero \ne \bsh \in \Lambda(\bsz,n)^{\perp}}
 \sum_{\bssigma \in \{\pm 1\}^d}  \hat{f}(|\bssigma(\bsh) + \bsk|)
 \frac{(\sqrt{2})^{-|\bssigma(\bsh) + \bsk|_0 +|\bsk|_0}}{2^{d}}\bigg|^2 \nonumber\\
 &\,\le\, \bigg(\sum _{\bszero \ne \bsh \in \Lambda(\bsz,n)^{\perp}}
 \sum _{\bssigma \in \{\pm 1\}^d}\frac{2^{-|\bssigma(\bsh) + \bsk|_0}}{2^d}\,r_{\alpha,\bsgamma}(\bssigma(\bsh) + \bsk)
 \left|\hat{f}(|\bssigma(\bsh) + \bsk|)\right|^2\bigg) \nonumber\\
 &\qquad\times\bigg(\sum_{\bszero \ne \bsh \in \Lambda(\bsz,n)^{\perp}}
 \sum_{\bssigma \in \{\pm 1\}^d}\frac{2^{|\bsk|_0}}{2^d\,r_{\alpha,\bsgamma}(\bssigma(\bsh) + \bsk)}\bigg).
\end{align}
The first factor in \eqref{factors} can be bounded from above by
relaxing the condition on $\bsh$ and instead summing over all
$\bsh\in\Z^d$:
\begin{align*}
 &\sum_{\bssigma \in \{\pm 1\}^d} \sum_{\bsh \in \Z^d}
 \frac{2^{-|\bssigma(\bsh) + \bsk|_0}}{2^d}r_{\alpha,\bsgamma}(\bssigma(\bsh) + \bsk)\left|\hat{f}(|\bssigma(\bsh) + \bsk|)\right|^2 \\
 &\,=\, \sum_{\bsh \in \Z^d} 2^{-|\bsh |_0}\,r_{\alpha,\bsgamma}(\bsh )\left|\hat{f}(|\bsh|)\right|^2
 \,=\, \sum_{\bsh \in \Z^d_+}r_{\alpha,\bsgamma}(\bsh )\left|\hat{f}(\bsh)\right|^2
 \,=\, \norm{f}^2_{C_{d,\alpha,\bsgamma}},
\end{align*}
where the first equality holds since the vectors $\bssigma(\bsh)+\bsk$ are
precisely all of $\Z^d$ as we sum over all $\bsh\in\Z^d$. These
bounds for the two sources of errors lead to the worst-case error bound in
the theorem.
\end{proof}

\subsection{Connection with the weighted Korobov space}

In \cite{KSW2006}, a similar approximation algorithm was developed for the
weighted Korobov space with rank-$1$ lattice points instead of
tent-transformed rank-$1$ lattice points. It was shown there that the
squared worst-case error for the corresponding algorithm, which we denote
by $\widetilde{A}_{n,d,M}$ here, is bounded by
\begin{align}
 \label{WCE_bound_KSW}
 e^{\mathrm{wor}}(\widetilde{A}_{n,d,M}; E_{d,\alpha,\bsgamma})^2
 \,\le\, \frac{1}{M} + \widetilde{E}_{n,d,M}(\bsz),
\end{align}
with
\begin{align}
 \label{E_KSW}
 \widetilde{E}_{n,d,M}(\bsz) \,:=\,
 \sum_{\bsk \in \widetilde{H}_M}\sum_{\bszero \ne \bsh \in \Lambda(\bsz,n)^{\perp}}
 \frac{1}{r_{\alpha,\bsgamma}(\bsh + \bsk)},
\end{align}
where $\widetilde{H}_M$ denotes the hyperbolic cross used with the Fourier
basis functions. The set $\widetilde{H}_M$ differs from $H_M$ in that it is defined over all the
hyperoctants instead of just the positive hyperoctant:
\begin{align*}
 \widetilde{H}_M \,:=\, \widetilde{H}_M^{d,\alpha,\bsgamma}
 \,:=\, \{\bsk \in \mathbb{Z}^d :  r_{\alpha,\bm{\gamma}}(\bsk)\leq M \}.
\end{align*}

\begin{theorem}
The bounds of the squared worst-case errors in \eqref{WCE_bound} and
\eqref{WCE_bound_KSW}--\eqref{E_KSW} are equal, given that the weights
$\bsgamma$ and the decay parameters $\alpha$ of the two spaces are the
same:
\begin{align} \label{eq:same}
 \sum_{\bsk \in H_M}\sum_{\bszero \ne \bsh \in \Lambda(\bsz,n)^{\perp}}
 \sum_{\bssigma \in \{\pm 1\}^d}\frac{2^{|\bsk|_0}}{2^d\,r_{\alpha,\bsgamma}(\bssigma(\bsh) + \bsk)}
 &\,=\,
 \sum_{\bsk \in \widetilde{H}_M}\sum_{\bszero \ne \bsh \in \Lambda(\bsz,n)^{\perp}}
 \frac{1}{r_{\alpha,\bsgamma}(\bsh + \bsk)}.
\end{align}
\end{theorem}

\begin{proof}
Starting with the left-hand side of \eqref{eq:same}, we can write
\begin{align*}
 \mbox{LHS of \eqref{eq:same}}
 &\,=\, \sum_{\bsk \in H_M}\sum_{\bszero \ne \bsh \in \Lambda(\bsz,n)^{\perp}}\sum_{\bssigma \in \{\pm 1\}^d}
 \frac{2^{|\bsk|_0}}{2^d\,r_{\alpha,\bsgamma}(\bssigma(\bsh + \bssigma^{-1}(\bsk)))}\\
 &\,=\, \sum_{\bsk \in H_M}\sum_{\bszero \ne \bsh \in \Lambda(\bsz,n)^{\perp}}\sum_{\bssigma^{-1} \in \{\pm 1\}^d}
 \frac{2^{|\bsk|_0}}{2^d\,r_{\alpha,\bsgamma}(\bsh + \bssigma^{-1}(\bsk))}
 \,=\, \mbox{RHS of \eqref{eq:same}},
\end{align*}
where in the second equality we have used the sign invariance of
$r_{\alpha,\bsgamma}$ and the fact that for any function $G\colon
\Z^d\nach \C$ and any $\bsk \in \Z^d$,
\[
  \sum_{\bssigma \in \{\pm 1\}^d}G(\bssigma^{-1}(\bsk))=\sum_{\bssigma^{-1} \in \{\pm 1\}^d}G(\bssigma^{-1}(\bsk)),
\]
as it is the same sum in a different order. For the final equality we have
used the identity in \eqref{identity_2} but with the sets $\Z_+^d$ and
$\Z^d$ replaced by $H_M$ and $\widetilde{H}_M$. The result is hence
proved.
\end{proof}

The quantity $\widetilde{E}_{n,d,M}(\bsz)$ in
\eqref{WCE_bound_KSW}--\eqref{E_KSW} was used in \cite{KSW2006} as the
search criterion in a component-by-component search algorithm to construct
the generating vector for a lattice rule that satisfies a proven
worst-case error bound for approximation in the Korobov space. We see from the
above theorem that this quantity coincides with the second term in
\eqref{WCE_bound}. Thus the generating vector constructed by the algorithm
for the Korobov space can also be used in a tent-transformed lattice rule for
approximation in the cosine space. A fast implementation of this
construction in the spirit of \cite{NC06a} is discussed in
\cite{KSW08}. We summarize this conclusion in the corollary below.

\begin{corollary} \label{error-bound3}
Let $\kappa>1$ be some fixed number and suppose $n$ is a prime number
satisfying $n \ge \kappa M^{1/(2\alpha)}$. A fast
component-by-component search algorithm can be used to obtain a generating
vector $\bsz\in \Z^d_n$ in $\calO(|\widetilde{H}_M|\,d\, n \,\log n)$
operations, using the expression $\widetilde{E}_{n,d,M}(\bsz)$ in
\eqref{E_KSW} as the search criterion, such that the worst-case error in
the cosine space $C_{d,\alpha,\bsgamma}$ for the algorithm $A_{n,d,M}$
defined by \eqref{approx_alg} with the resulting tent-transformed lattice
points satisfies
\[
 e^{\mathrm{wor}}(A_{n,d,M}; C_{d,\alpha,\bsgamma})^2
 \,\le\, \frac{1}{M} + \frac{M^{\tau/\lambda}}{(n-1)^{1/\lambda}}
 \frac{1}{\mu}
 \prod_{j=1}^d \left[\left(1+2\zeta(2\alpha\tau)\gamma_j^\tau\right)
 \left(1+2(1+\mu^\lambda)\zeta(2\alpha\lambda)\gamma_j^\lambda\right)
 \right]^{1/\lambda}
\]
for all $\tau > 1/(2\alpha)$, $1/(2\alpha) < \lambda \le 1$, and $0 < \mu
\le (1-1/\kappa)^{2\alpha}$, where $\zeta(\cdot)$ again denotes the
Riemann zeta function. Hence, upon setting $\tau = \lambda$ and
choosing $M = \calO(n^{1/(2\lambda)})$ to balance the order of the two
error contributions, we conclude that the convergence rate is
$\calO(n^{-1/(4\lambda)})$, with the implied constant independent of $d$
if $\sum_{j=1}^\infty \gamma_j^\lambda < \infty$. As $\lambda\to
1/(2\alpha)$, the method achieves the convergence rate close to
$\calO(n^{-\alpha/2})$, while the optimal rate is believed to be close to
$\calO(n^{-\alpha})$.
\end{corollary}
\noindent
It was proved in \cite{KSW2006} that $|\widetilde{H}_M|
 \le M^q\prod_{j=1}^d(1+2\zeta(2\alpha q)\gamma_j^q)$ for all $q >1/(2\alpha)$, 
and this quantity can be bounded independently of $d$ if  $\sum_{j=1}^\infty \gamma_j^q <\infty.$

Tractability analysis for approximation in the cosine space can be carried
out following exactly the same argument as in \cite{KSW2006} for the
Korobov space. Roughly speaking, tractability means that the minimal
number of function evaluations required to achieve an error $\varepsilon$
in $d$ dimensions is bounded polynomially in $\varepsilon^{-1}$ and $d$,
while strong tractability means that this bound is independent of $d$.
Tractability depends on the problem setting and on the type of information
used by algorithms, see the books \cite{NW08,NW10,NW12}. In the
corollary below we provide just an outline of tractability results for
approximation in the cosine space.

\begin{corollary}
Consider the approximation problem for weighted cosine spaces in the
worst-case setting.
\begin{enumerate}
\item[\textnormal{(a)}] Let
    $p^*=2\max\left(\tfrac{1}{2\alpha},s_{\bsgamma}\right)$, where
    $s_{\bsgamma}=\inf\{ s>0 : \sum_{j=1}^\infty \gamma_j^s <
    \infty\}$ and suppose that
\[
  \sum_{j=1}^\infty\gamma_j<\infty.
\]
Given $\varepsilon > 0$, the approximation algorithm $A_{n,d,M}$
defined by \eqref{approx_alg}, with appropriately chosen values of $n$
and $M$ and specially constructed generating vector $\bsz$, achieves
the error bound $e^{\mathrm{wor}}(A_{n,d,M}; C_{d,\alpha,\bsgamma})
\le \varepsilon$ using $n=\calO(\varepsilon^{-p})$ function values.
The implied factor in the big $\calO$ notation is independent of $d$
and the exponent $p$ is arbitrarily close to $2p^*$.
\item[\textnormal{(b)}] Suppose that
\[
  a := \limsup_{d\to\infty} \frac{\sum_{j=1}^d \gamma_j}{\log(d+1)} < \infty.
\]
Given $\varepsilon > 0$, the approximation algorithm $A_{n,d,M}$
defined by \eqref{approx_alg}, with appropriately chosen values of $n$
and $M$ and specially constructed generating vector $\bsz$, achieves
the error bound $e^{\mathrm{wor}}(A_{n,d,M}; C_{d,\alpha,\bsgamma})
\le \varepsilon$ using $n= \calO(\varepsilon^{-4} d^{\,q})$ function
values. The implied factor in the big $\calO$ notation is independent
of $\varepsilon$ and $d$, and the exponent $q$ can be arbitrarily
close to $4\zeta(2\alpha)a$.
\end{enumerate}
\end{corollary}

\section{Conclusions} \label{sec:conc}

We have studied the problems of integration and approximation in the
weighted cosine space of smooth non-periodic functions using
tent-transformed lattice points. For the integration problem, we provided
a precise formula for the squared
worst-case error of a tent-transformed lattice rule, amending the result in
\cite{DNP2012}. We also derived the
root-mean-squared worst-case error of a tent-transformed randomly-shifted
lattice rule. By exploiting the connection with the weighted Korobov space of
smooth periodic functions, we show that these methods can be constructed
to achieve the optimal rate of convergence in the cosine space. For the
approximation problem, we showed that the worst-case error for our
algorithm in the cosine space has an upper bound which is identical to a
previously analyzed upper bound on the worst-case error for a related
algorithm in the Korobov space, and this allowed us to apply known
constructive results for approximation in the Korobov space to the cosine
space.

\section*{Acknowledgements}

We graciously acknowledge the financial supports from the Australian
Research Council (FT130100655 and DP150101770) and the KU Leuven research
fund (OT:3E130287).

\bibliographystyle{abbrv}
\bibliography{biblio}
\end{document}